\documentclass{amsart}
\usepackage{amssymb}
\usepackage{amsthm}

\newtheorem{prop}{Proposition}

\DeclareMathAlphabet{\mathchorus}{OT1}{cmtt}{m}{n}  

\newcommand{\mc}[1]{\mathchorus{#1}} 

\newcommand{\bv}{\begin{pmatrix}}  
\newcommand{\ev}{\end{pmatrix}} 
\newcommand*{\defeq}{\stackrel{\text{def}}{=}} 
\newcommand*{\cat}{\mkern 2mu _{\smallsmile}} 

\begin{document}
\title{A remark on the sequence defined by the nonhomogeneous linear difference equation}
\author{Sergei Kazenas }
\begin{abstract}
In this short paper, a formula for the sequence defined by the nonhomogeneous linear difference equation with variable coefficients is presented. A connection with the homogeneous case is shown.
\end{abstract}

\email{kazenas@protonmail.com} \maketitle

Consider the sequence $\bv w_1, w_2, ... \ev$ recursively defined by the equation:
$$
w_n = c_n + \sum_{j=0}^{n-1}a_{n,j}w_j  \text{,}
\eqno{(1)}$$
where $w_0$ is an arbitrary number; $\bv c_1, c_2, ...\ev$ and $\bv a_{1,0},a_{2,0},a_{2,1}, ...\ev$ are arbitrary number sequences.
A similarly defined recursive sequence was studied by Mallik in \cite{Mal1998}, for example. 
Here, the reasoning that was introduced in \cite{Kaz2019} is  utilized.

In order to present the result, it is convenient to treat the coefficients $a_{n,j}$ as values of a function $a_{x,y}$ such that $a_{n,j} = a_{x,y} |_{x=n, y=j}$.

\begin{prop}
If $w_n$ is defined by \normalfont{(1)}, then
$$
\sum_{j=0}^n w_j = c_n + w_0 \Phi_n a_{x,y} + \sum_{\ell=1}^{n-1} c_\ell \Phi_{n-\ell} a_{x+\ell,y+\ell}
\text{,}
\eqno{(2)}$$
where $\Phi_n = \Phi_{n;x,y}$ are operators that act on space of functions of two variables and defined by the rule:
\begin{align}
\Phi_{n;x,y}f_{x,y}\thinspace &  \normalfont{\defeq} \medspace \sum_{j=1}^{2^n} \prod_{k=0}^{n-1}(1+[2^k,2^k]_j (f_{k+1,\log_2\lceil 1+(j-1)\bmod2^k \rceil}-1)) \notag \\
                      &= 1 + \sum_{m=1}^n \sum_{1\leqslant k_1 < \ldots < k_m \leqslant n} f_{k_1,0}\prod_{i=1}^m f_{k_i,k_{i-1}} \text{.} \notag
\end{align}
\end{prop}
\begin{proof}
Using the reasoning and the lemma from \cite{Kaz2019}, we construct vectors $\mc{w}_n $  such that $|\mc{w}_n|_1 = \sum_{j=0}^n w_j$:
\begin{align}
\mc{w}_n &= \mc{u}_n (\mc{1}_2 \times \mc{w}_{n-1}) + \mc{c}_n \notag \\
         &=(\mc{1}_{2^{n-1}} \times \mc{w}_1) \prod_{k=2}^n(\mc{1}_{2^{n-k}} \times \mc{u}_k)+
\sum_{i=2}^n(\mc{1}_{2^{n-i}} \times \mc{c}_i)\prod_{k=i+1}^n(\mc{1}_{2^{n-k}} \times \mc{u}_k) \text{,} \notag
\end{align}
where 
$\mc{w}_1   = \bv w_0, c_1 + a_{1,0}\ev $,
$\mc{c}_i   = c_i (\mc{0}_{2^n-1} \cat \mc{1}_1)$,
$\mc{u}_k   = \mc{1}_{2^{k-1}} \cat \bv a_{k,\lceil \log_2 i \rceil}\ev_{i=1}^{2{k-1}}$.

~\

Next, we should replace each of the expressions $\mc{1}_{2^{n-1}}$, $\mc{1}_{2^{n-k}}$, $\mc{1}_{2^{n-i}}$ by $\mc{1}_{\infty}$ 
and take the first $2^n$ elements of the resulting vector.
Finally, the transition to functions and the use of some combinatorics complete the proof.
\end{proof}

\emph{Remark}. It is easy to check that the values $\Phi_{n-\ell} a_{x+\ell,y+\ell}$ (2) can be treated as
the solutions of the homogeneous linear difference equation with the "shifted" coefficients.

\emph{Example}. Define $v_n$ by the equation: $v_0 = 1$, $v_n = \sum_{j=0}^{n-1}a_{n,j}v_j$ (for $n \geqslant 1$).

The first terms of the sequence are as follows:
\begin{align}
v_1 &= a_{1,0}, \notag \\
v_2 &= a_{2,0}+a_{1,0} a_{2,1} , \notag \\
v_3 &= a_{3,0}+a_{1,0} a_{3,1}+a_{2,0} a_{3,2}+a_{1,0} a_{2,1} a_{3,2}, \notag \\
v_4 &= a_{4,0}+a_{1,0} a_{4,1}+a_{2,0} a_{4,2}+a_{1,0} a_{2,1} a_{4,2}+a_{3,0} a_{4,3}+a_{1,0} a_{3,1} a_{4,3}+a_{2,0} a_{3,2} a_{4,3} \notag \\
    & + a_{1,0} a_{2,1} a_{3,2}a_{4,3}. \notag
\end{align}
Now by "shifting" coefficients we can write out:
\begin{align}
w_4 &= c_4 \notag \\
    &+ a_{4,3} c_3 \notag \\ 
    &+ (a_{4,2}+a_{3,2} a_{4,3}) c_2 \notag \\
    &+ (a_{4,1}+a_{2,1} a_{4,2}+a_{3,1} a_{4,3}+a_{2,1} a_{3,2} a_{4,3})c_1 \notag \\
    &+ (a_{4,0}+a_{1,0} a_{4,1}+a_{2,0} a_{4,2}+a_{1,0} a_{2,1} a_{4,2}+a_{3,0} a_{4,3}+a_{1,0} a_{3,1} a_{4,3}+a_{2,0} a_{3,2} a_{4,3} \notag \\
    &+ a_{1,0} a_{2,1} a_{3,2}a_{4,3})w_0. \notag
\end{align}

\vspace{5mm}
\end{document}